\documentclass[11 pt, final]{article}
\title{Translation-like actions of nilpotent groups}
\author{David Bruce Cohen and Mark Pengitore}

\usepackage{fullpage}
\usepackage[latin1]{inputenc}
\usepackage{tikz}
\usetikzlibrary{shapes,arrows}
\usepackage{graphicx}
\usepackage{amsmath}
\usepackage{amssymb}
\usepackage{amsthm}
\usepackage{epsfig,pinlabel}
\usepackage{amsfonts}
\usepackage{amscd}
\usepackage[all,tips]{xy}

\DeclareMathOperator{\ZZ}{\ensuremath{\mathbb{Z}}}
\DeclareMathOperator{\NN}{\ensuremath{\mathbb{N}}}

\DeclareMathOperator{\RR}{\ensuremath{\mathbb{R}}}

\DeclareMathOperator{\diam}{diam}
\DeclareMathOperator{\cone}{Cone}

\def\To{{\rightarrow}}
\def\mfu{{\mathfrak{u}}}

\def\mfU{{\mathfrak{U}}}
\def\heps{{h_\epsilon}}
\def\meps{{\mu(\epsilon)}}
\def\hneps{{h_{\epsilon,n}}}
\def\Aeps{{A_\epsilon}}
\def\Aneps{{A_{\epsilon,n}}}
\def\finf{{f_\infty}}

\theoremstyle{plain}
\newtheorem{theorem}{Theorem}[section]
\newtheorem{lemma}[theorem]{Lemma}

\newtheorem{defn}[theorem]{Definition}
\newtheorem{prop}[theorem]{Proposition}

\newtheorem{ex}[theorem]{Example}

\newtheorem{corollary}[theorem]{Corollary}
\theoremstyle{definition}
\newtheorem*{definition}{Definition}

\DeclareMathOperator{\Log}{Log}

\newcommand{\eps}{\varepsilon}

\newcommand{\bdef}{\overset{\text{def}}{=}}

\newcommand{\ga}{\gamma}
\newcommand{\Ga}{\Gamma}

\newcommand{\De}{\Delta}

\newcommand{\set}[1]{\left\{#1\right\}}

\newcommand{\pr}[1]{\left( #1 \right) }

\newcommand{\Fr}[1]{\ensuremath{\mathfrak{#1}}}

\newcommand{\N}{\ensuremath{\mathbb{N}}}
\newcommand{\Q}{\ensuremath{\mathbb{Q}}}
\newcommand{\R}{\ensuremath{\mathbb{R}}}
\newcommand{\Z}{\ensuremath{\mathbb{Z}}}

\newcommand{\map}[3]{#1 : #2 \rightarrow #3}

\def\conga{{\text{Cone}(\Gamma)}}
\def\congahat{{\text{Cone}(\widehat{\Gamma})}}
\def\gahat{{\widehat{\Gamma}}}
\def\gainf{{\Gamma_\infty}}

\begin{document}

\maketitle

\begin{abstract}
We give a new obstruction to translation-like actions on nilpotent groups. Suppose we are given two finitely generated torsion free nilpotent groups with the same degree of polynomial growth, but non-isomorphic Carnot completions (asymptotic cones). We show that there exists no injective Lipschitz function from one group to the other. It follows that neither group can act translation-like on the other.
\end{abstract}

\section{Introduction}
\subsection{Translation-like actions.}
\begin{defn}
\cite[Definition 6.1]{Whyte1} Let $G$ and $H$ be finitely generated groups equipped with word metrics. We say that an action of $H$ on $G$ is free if no $g\in G$ can be fixed by any $h\in H\setminus\{1_H\}$. We say that an action of $H$ on $G$ is translation-like if it is free and, for every $h\in H$, we have $\sup_{g\in G}d(g,h\cdot g)<\infty.$ 
\end{defn}

Given a translation-like action of $H$ on $G$, one sees that for some Cayley graph of $G$, each orbit of $H$ is an embedded copy of the Cayley graph of $H$, and the disjoint union of all $H$-orbits covers $G$.  Any finitely generated group $G$ acts translation-like on itself by letting $g\in G$ act as the map
$$h\mapsto hg^{-1}.$$
By restricting this action, we see that any subgroup of $G$ acts translation-like on $G$. Thus, a translation-like action by $H$ is a geometric generalization of a $H$ subgroup.

\paragraph{Geometric analogues of conjectures.}There are many questions in group theory which ask whether having a subgroup of some isomorphism type is a complete obstruction to some property. For example, the Burnside problem asked whether having an infinite cyclic subgroup is a complete obstruction to finiteness. For another example, the Von Neumann-Day problem asked whether having a $\ZZ\ast\ZZ$ subgroup is a complete obstruction to amenability. Both of these questions have negative answers, by work of Golod-Shafarevich \cite{gs} and Olshanskii \cite{olshanskii}, respectively. However, if we ask instead whether having a translation-like action by $\ZZ$ (respectively $\ZZ\ast\ZZ$) is a complete obstruction to finiteness (respectively amenability), the answer is known to be positive, as we explain below.

\paragraph{Examples of translation-like actions.}
We now give some examples of translation-like actions that do not arise from subgroups.
\begin{itemize}
\item If there is a bilipschitz map $\Psi:G\To H,$ then $H$ acts translation like on $G$ by setting $h\in H$ to act as the map
$$g\mapsto \psi^{-1}(\psi(g)h^{-1}).$$
\item Seward \cite{Seward1} showed that $\ZZ$ acts translation-like on every infinite group, or equivalently that having a translation-like action of $\ZZ$ is a complete obstruction to being finite .
\item Whyte \cite{Whyte1} showed that $\ZZ\ast\ZZ$ acts translation-like on every non-amenable group.
\item The first author \cite{Cohen1} showed that $\ZZ\times\ZZ$ acts translation-like on the fundamental group of a closed hyperbolic $3$-manifold.
\end{itemize}

\paragraph{Known obstructions to translation-like actions.} We briefly survey some known obstructions to the existence of a translation-like action.
\begin{itemize}
\item Whyte \cite{Whyte1} showed that $\ZZ\ast\ZZ$ cannot act translation-like on an amenable group. Since it acts translation-like on every non-amenable group (as noted above), this shows that having a translation-like action by $\ZZ\ast\ZZ$ is a complete obstruction to non-amenability. Note that many amenable groups (e.g., non-nilpotent elementary amenable groups \cite{chou}) admit Lipschitz injections from $\ZZ\ast\ZZ$, so admitting a translation-like action by some group is really a stronger notion than admitting a Lipschitz injection from that group.
\item If the growth function of $H$ has a greater growth rate than $G$, then $H$ cannot act translation-like on $G$. In particular, if $G$ is nilpotent, then any group acting translation-like on $G$ must also be nilpotent. 
\item If the asymptotic dimension of $H$ is greater than $G$, then $H$ cannot act translation-like on $G$.\cite{jiang}\cite[\S 6]{bst}
\item If the separation function of $H$ is greater than $G$, then $H$ cannot act translation-like on $G$.\cite{jiang}\cite[Lemma 1.3]{bst}
\end{itemize}

We mention also work of Jeandel which has shown that certain dynamical properties which always pass from a group $G$ to its subgroups also pass to any $H$ which acts translation-like on $G$.

\subsection{Nilpotent groups.}
Recall that a finitely generated group has polynomial growth if and only if it is virtually nilpotent \cite{Bass01},\cite[Theorem 3.2]{wolf},\cite{gromov}.  If a group $H$ acts translation-like on a nilpotent group $G$, it follows that $H$ has polynomial growth bounded by that of $G$, and hence must be virtually nilpotent itself. Separation function and asymptotic dimension give additional obstructions to the existence of translation-like actions on a nilpotent group, but these obstructions are asymmetric (if one of them obstructs $H$ from acting on $G$, then it does not necessarily obstruct $G$ from acting on $H$) and we know of no algorithm for computing them.

A result of Pansu \cite{Pansu} asserts that the asymptotic cone of a finitely generated, torsion free nilpotent group $\Gamma$ may be naturally identified with a nilpotent Lie group (see \S\ref{subsection:carnotgroups}) known as the Carnot completion of $\Gamma$. The following is our main theorem.

\begin{theorem}
\label{introtheorem}
Suppose $\Ga$ and $\De$ are finitely generated, torsion free nilpotent groups of the same polynomial degree of growth, and $f:\Ga\To\De$ is an injective Lipschitz map. Then $\Ga$ and $\De$ have asymptotic cones which are isomorphic as Lie groups.

In particular, if finitely generated torsion free nilpotent groups $\Ga$ and $\De$ have the same polynomial degree of growth but non-isomorphic asymptotic cones, then neither group can act translation-like on the other.
\end{theorem}

For us, $f$ being Lipschitz means that for some constant $C$, $d(f(x),f(y))\leq Cd(x,y)$ for all $x,y\in\Ga$, where distances are measured with respect to some word metrics on $\Ga$ and $\De$. The result on translation-like actions follows immediately from the result on Lipschitz injections because, given such an action, the orbit map $g\mapsto g\cdot 1_\De$ Lipschitz injects $\Ga$ into $\De$. Theorem \ref{introtheorem} is proved in \S\ref{subsection:maintheorem} as Theorem \ref{maintheorem}.

\paragraph{Lipschitz embeddings.} We shall now briefly outline the proof of our theorem, and indicate the importance of our hypothesis on growth. First, let us consider the closely related problem of finding obstructions to the existence of Lipschitz \textit{embeddings} from a finitely generated nilpotent group $\Ga$ to another such group $\De$---i.e., Lipschitz maps $f:\Ga\To\De$ satisfying a lower bound of the form $d(f(x),f(y))\geq Cd(x,y)$. Li \cite[Theorem 1.4]{li} shows that if such an $f$ exists, then there is a homomorphic embedding of the asymptotic cone of $\Ga$ into the asymptotic cone of $\De$. To see that this is true, observe that such a map $f$ would induce a Lipschitz embedding $f_\infty$ of the asymptotic cone of $\Ga$ into the asymptotic cone of $\De$. This $f_\infty$ would be Pansu differentiable almost everywhere (see \S\ref{subsection:pansuderivative}), and the Pansu derivative would yield an injective homomorphism $\Ga$ to $\De$.

\paragraph{Lipschitz injections.} The hypotheses of our theorem, however, do not give us a Lipschitz embedding, but a mere Lipschitz \textit{injection} $f:\Ga\To\De$---the only bound on the distortion of $f$ is that $d(f(x),f(y))>0$ when $d(x,y)>0$. Since $f$ is Lipschitz, it still induces a Lipschitz map $\finf$ on asymptotic cones, but a priori we do not know that the Pansu derivative of $\finf$ is injective, or even nontrivial. Indeed, Assouad's Theorem \cite[Proposition 2.6]{assouad} implies that every finitely generated nilpotent group can be Lipschitz injected into some $\ZZ^n$, but the asymptotic cone of $\Ga$ cannot be homomorphically embedded in that of $\ZZ^n$ unless $\Ga$ is virtually abelian. That indicates the necessity of some additional hypothesis.

\paragraph{The growth hypothesis.} Under our hypothesis that $\Ga$ and $\De$ have the same polynomial growth rate, we will see that the Pansu derivative of the induced map on asymptotic cones must be a group isomorphism. In particular, we will see that if it is not an isomorphism, then its image is killed by some nonzero homomorphism $\ell$ from the cone of $\De$ to $\RR$ (Corollary \ref{annihilating_functional}). That will imply for every $N>0$ there exists balls $B\subset\Ga$ and $B'\subset\De$ having comparable radius, such that we may fit at least $N$ disjoint translates of $f(B)$ inside $B'$---such translates may be found by moving ``perpendicular" to the kernel of $\ell$. Since $f$ is injective, the union of these translates will have cardinality $N\# f(B)$, contradicting our assumption on growth. 

\subsection{Acknowledgments.} We wish to thank Moon Duchin and Xiongdong Xie for conversations, and Yongle Jiang for sharing an early version of his work with us. We especially wish to thank Robert Young, who suggested that we look at the Pansu derivative, and told us about Assouad's theorem. The second author would like thank his advisor Ben McReynolds for all his help and guidance. The first author has been supported by NSF award 1502608.

\section{Background}

Let $f,g:\N\To\N$ be non-decreasing functions. We say that $f \preceq g$ if there exists a constant $C > 0$ such that $f(n) \leq C  g(Cn)$ for all $n \in \N$. We say that $f \approx g$ if $f \preceq g$ and $g \preceq f$. We say $f(n)=O(g(n))$ if there is a constant $C$ such that $f(n)\leq Cg(n)$ and $f(n)=o(g(n))$ if $\frac{f(n)}{g(n)}$ goes to $0$ as $n\To \infty$.

Let $\Ga$ be a finitely generated group. We define the commutator of $g, h \in \Ga$ as $[g,h] = g \: h \: g^{-1} \: h^{-1}$. If $A,B \leq \Ga$, we define the commutator of $A$ and $B$ as $[A,B] = \set{[a,b] \: | \: a \in A \text{ and } b \in B}.$ We define the abelianization of a group as $\Ga_{Ab} \bdef \Ga / [\Ga,\Ga]$.

\subsection{Nilpotent groups and nilpotent Lie algebras}

In this section we will review some basic notions of the theory of nilpotent groups.  In particular, we will define the rank, the Mal'tsev completion, and the growth of a nilpotent group.

We define the \emph{$i$-term of the lower central series} in the following way. We let $\Ga_1 \bdef \Ga$ and then for $i>1$ define $\Ga_i \bdef [\Ga,\Ga_{i-1}]$. 
\begin{defn}
	We say that $\Ga$ is \emph{nilpotent of step size} $c$ if $c$ is the minimal natural number such that $\Ga_{c+1} = \set{1}$. If the step size is unspecified, we just say that $\Ga$ is a nilpotent group.
\end{defn}
There is a natural notion of dimension for a torsion free, finitely generated nilpotent group. We define the \emph{rank of $\Ga$} as
$$
\text{rank}(\Ga) = \sum_{i=1}^{c} \text{rank}_\Z(\Ga_i / \Ga_{i+1}).
$$

Let $\Fr{g}$ be a finite dimensional $\R$-Lie algebra. The $i$-th term of the \emph{lower central series of $\Fr{g}$} is defined by $\Fr{g}_1 \bdef \Fr{g}$ and $\Fr{g}_i \bdef [\Fr{g},\Fr{g}_{i-1}]$ for $i > 1$.
\begin{defn}
	We say that $\Fr{g}$ is a nilpotent Lie algebra of step length $c$ if $c$ is the minimal natural number such that $\Fr{g}_{c+1} = \set{0}$. If the step size is unspecified, we just say that $\Fr{g}$ is a nilpotent Lie algebra.
\end{defn}

For a connected, simply connected nilpotent Lie group $G$ with Lie algebra $\Fr{g}$, the exponential map, written as $\map{\exp}{\Fr{g}}{G}$ (see \cite[Theorem 1.127]{Knapp}) is a diffeomorphism whose inverse is formally denoted as $\Log$. By \cite[Theorem 7]{Maltsev_complete}, $G$ admits a cocompact lattice $\Ga$ if and only if $\Fr{g}$ admits a basis $\set{X_i}_{i=1}^{\text{dim}(G)}$ with rational structure constants. We then say that $G$ is $\Q$-defined.  For any torsion-free, finitely generated nilpotent $\Ga$, \cite[Theorem 6]{Maltsev_complete} implies there exists a unique up to isomorphism $\Q$-defined nilpotent group such that $\Ga$ embeds as a cocompact lattice. 

\begin{defn}
	We call this $\Q$-defined Lie group the \emph{Mal'tsev completion of $\Ga$} and denote it as $\widehat{\Ga}$.
\end{defn}

We have the following examples of the Mal'tsev completion of nilpotent groups
\begin{ex}
	The Mal'tsev completion of $\Z^k$ is $\R^k$.
\end{ex}
\begin{ex}
	We establish some notation for the following example. For a commutative unital ring $R$, we define $H_3(R)$ to be the group of $3 \times 3$ upper triangular matrices with $1's$ on the diagonal and $R$-valued coefficients. We then have that $H_3(\Z)$ is the $3$-dimensional integral Heisenberg group where $\widehat{H_3(\Z)} \cong H_3(\R)$.
\end{ex}

We have that the nilpotent step length of $\Ga$ is equal to the step length of $\widehat{\Ga}$. Moreover, $\text{rank}(\Ga) = \text{dim}(\widehat{\Ga})$. Finally, we have that $\widehat{\Ga}_i$ is the Mal'tsev completion of $\Ga_i$. See \cite{Dekimpe}
for more details about the Mal'tsev completion of a torsion-free finitely generated nilpotent group.
	
\paragraph{Growth rates.} Let $\Ga$ be a finitely generated group with a symmetric finite generating subset $S$. We define $\ga_{\Ga}^{S}(n) = |B_{\Ga,S}(n)|$. One observation is that if $S_1$ and $S_2$ are different symmetric finite generating subsets, then $\ga_{\Ga}^{S_1}(n) \approx \ga_{\Ga}^{S_2}(n)$. We refer to the equivalence class of $\ga_{\Ga}^{S}$ as the \emph{growth rate of $\Ga$.} 

When $\Ga$ is an infinite, finitely generated nilpotent group of step size $c$ with a symmetric generating subset $S$, then $\ga_{\Ga}^S(n) \approx n^{d(\Ga)}$ where $d(\Ga) \in \N$ \cite{Bass01}. We call $d(\Ga)$ the \emph{homogeneous dimension of} $\Ga$ and \cite{Bass01} gives a precise computation of $d(\Ga)$ as
$$
d(\Ga) = \sum_{k=1}^{c} k \: \text{dim}_\Z(\Ga_k / \Ga_{k+1}). 
$$

\subsection{The asymptotic cone}
We will now use nonstandard analysis to define an object known as the asymptotic cone of a metric space.  For background on nonstandard analysis see \cite{robinson}, and for a more detailed description of the asymptotic cone see \cite{drutu}. For the duration of this paper, fix a nonprincipal ultrafilter $\mfU$.  Given a sequence of real numbers $(x_n)_{n\in\NN}$, we write
$$\lim_\mfU x_n = x$$
if there is a $U\in\mfU$ such that for all $\epsilon>0$, there exists an $N$ with $|x_n-x|<\epsilon$ for all $n\in U$ such that $n\geq N$. Every sequence has at most one such limit, and if it has no such limit, then it converges to $\infty$ or $-\infty$ along $\mfU$. Given a compact metric space $K$ and a sequence $(x_n)_{n\in\NN}$ of elements of $K$, we define $\lim_\mfU x_n$ to be the unique element $x$ of $K$ such that $\lim_\mfU d(x,x_n)=0$.

\begin{definition}[Asymptotic cone]
Let $X$ be a metric space, $(r_n)_{n\in\NN}$ a sequence of positive real numbers going to $\infty$, and $(b_n)_{n\in\NN}$ a sequence of elements of $X$. Consider the set
$$X[\mfU,(r_n)_{n\in\NN},(b_n)_{n\in\NN}]: \bdef \left\{(x_n)_{n\in\NN}:x_n\in X; \lim_\mfU \frac{d(x_n,b_n)}{r_n}<\infty\right\}
$$
equipped with the pseudometric
$$d((x_n)_{n\in\NN},(y_n)_{n\in\NN})=\lim_\mfU \frac{d(x_n,y_n)}{r_n}.$$
The metric space obtained from $X[\mfU,(r_n),(b_n)]$ by identifying $x$ and $y$ whenever $d(x,y)=0$ is called the \textit{asymptotic cone} of $X$ with respect to $\mfU,(r_n)$, and $(b_n)$. Given a sequence $(x_n)_{n\in\NN}$ of elements of $X$, we will write $[x_n]$ for the corresponding point of the asymptotic cone.
\end{definition}

\paragraph{Standing assumptions.} We do not require the full generality of the above definition. When asymptotic cones are constructed in this paper, we will make the following assumptions.
\begin{itemize}
\item The metric space $X$ whose asymptotic cone we wish to construct is a finitely generated group $\Gamma$ equipped with a left invariant metric.
\item The sequence of base points $(b_n)_{n\in\NN}$ has all terms equal to the identity $1_\Gamma$.
\item The sequence of scaling factors $(r_n)_{n\in\NN}$ is given by $r_n=n$.
\end{itemize}

\begin{definition}
Under these assumptions we denote the asymptotic cone of $\Gamma$ as $\cone(\Gamma)$.
\end{definition}

\subsection{Carnot Groups}
\label{subsection:carnotgroups}
Let $\Gamma$ be a torsion free, finitely generated nilpotent group. Pansu \cite{Pansu} showed that the metric space $\cone(\Gamma)$ is given by a deformation $\Gamma_\infty$ of $\widehat{\Gamma}$ known as the Carnot completion of $\Gamma$. This space $\Gamma_\infty$ is a nilpotent Lie group equipped with a Carnot-Carath\'{e}odory metric. A strengthening of Pansu's result, due to Cornulier \cite{cornulier}, shows that $\cone(\Gamma)$ has a natural group structure, and that there is an isometric isomorphism $\cone(\Gamma)\To\Ga_\infty$. We shall now give a description of $\Ga_\infty$. First, we define a Carnot group.

\begin{defn}
	Let $\Fr{g}$ be a nilpotent Lie algebra of step length $c$. We say that $\Fr{g}$ is a \emph{Carnot Lie algebra} if it admits a grading
	$
	\Fr{g} = \bigoplus_{i=1}^c \Fr{v}_i
	$
	where 
	$
	\Fr{g}_t = \bigoplus_{i=t}^c \Fr{v}_i
	$ and $\Fr{v}_1$ generates $\Fr{g}$. We say that Lie group $G$ is \emph{Carnot} if its Lie algebra is Carnot.
\end{defn}

\paragraph{The Carnot completion.} To a torsion free, finitely generated nilpotent group $\Ga$ of step length $c$, we associate a Carnot Lie group $\Gamma_\infty$ in the following way. Let $\Fr{g}$ be the Lie algebra of $\widehat{\Ga}$ and take 
$
\Fr{g}_{\infty} \bdef \bigoplus_{i=1}^{c} \Fr{g}_i / \Fr{g}_{i+1}.
$
 Since $[\Fr{g}_i,\Fr{g}_{j}] \subseteq \Fr{g}_{i+j}$, the Lie bracket of $\Fr{g}$ defines a bilinear map
$
\pr{\Fr{g}_i / \Fr{g}_{i+1}} \otimes \pr{\Fr{g}_j / \Fr{g}_{j+1}} \longrightarrow (\Fr{g}_{i+j} / \Fr{g}_{i+j+1})
$
 for any $1\leq i,j \leq c$. We extend this linearly to a Lie bracket $[ - : -]_\infty$ on 
$
\Fr{g}_{\infty}.
$
 The pair $\pr{\Fr{g}_\infty,[  - : - ]_\infty}$ is called the \emph{graded Lie algebra} associated to $\Fr{g}$. By exponentiating $\Fr{g}_\infty$, we obtain a connected, simply connected nilpotent Lie group $\Ga_\infty$ which we call the \emph{Carnot completion of $\Ga$.} By construction, we have that $\text{rank}(\Ga_i) = \text{dim}((\Ga_{\infty})_i)$ for all $i\geq 1$. In particular, the step length of $\Ga_{\infty}$ is equal to the step length of $\Ga$,  and $\text{dim}(\Ga_{\infty}) = \text{rank}(\Ga)$. 

\paragraph{Dilations of the Carnot completion.} Observe that the linear maps $\map{d\delta_t}{\Fr{g}_\infty}{\Fr{g}_\infty}$ given by
$$
d\delta_t(v_1,\cdots,v_c) = \pr{t \cdot v_1, t^2 \cdot v_2, \cdots, t^c \cdot v_c}.
$$
satisfy $d\delta_t([v,w]_\infty) = [d\delta_t(v),d\delta_t(w)]_\infty$ and $d\delta_{ts} = d\delta_t \circ d\delta_s$ for $v,w \in \Fr{g}_\infty$, $t,s > 0$.  Thus, $\set{d\delta_t \: | \: t > 0}$ gives a one parameter family of Lie automorphisms of $\Fr{g}_\infty$. Subsequently, we have an one parameter family of automorphisms of $\Ga_\infty$ denoted $\delta_t$. Since $\exp$ carries $\Fr{g}_2$ to $(\Ga_\infty)_2$, we see that $\exp$ induces an isomorphism of groups from $(\Fr{g}_\infty)_\text{Ab}$ to $(\Ga_\infty)_\text{Ab}$. If $\ell:\Gamma_\infty\To \RR$ is a homomorphism, there exists an induced map $\tilde{\ell}:(\Fr{g}_\infty)_\text{Ab}\To\RR$ such that $\ell\circ\exp(v_1,\ldots,v_c)=\tilde{\ell}(v_1)$. We see that
$$\ell(\delta_t(\exp(v_1,\ldots,v_c)))=\tilde{\ell}(tv_1)
=t\tilde{\ell}(v_1)=t\ell(\exp(v_1,\ldots,v_c)).$$

\paragraph{A Carnot-Carath\'{e}odory metric on the Carnot completion.} Fix a linear isomorphism $L:\Fr{g}\To\Fr{g}_\infty$ such that $L(\Fr{g}_i)=(\Fr{g}_\infty)_i$. Following \cite{cornulier}, equip $\widehat{\Ga}$ with the word metric associated to some compact generating set, and let $\Phi_n:\Ga_\infty\To\widehat{\Ga}$ be the function
$$\Phi_n:\Gamma_\infty
{\buildrel \text{log}\over\longrightarrow} \mathfrak{g}_\infty
{\buildrel \delta_{n}\over\longrightarrow} \mathfrak{g}_\infty
{\buildrel \text{L}\over\longrightarrow} \mathfrak{g}
{\buildrel \text{exp}\over\longrightarrow} \widehat{\Gamma},$$
and let $d_\mfU$ be the metric on $\Ga_\infty$ given by
$$d_\mfu(g,h)=\lim_\mfU \frac{d(\Phi_n(g),\Phi_n(h))}{n}.$$
Then we have that $d_\mfU$ is a left invariant Carnot-Carath\'{e}odory metric on $\Ga_\infty$ by \cite[Corollary A.10]{cornulier}, and that $|\delta_t(g)|=t|g|$ when $|\cdot|$ denotes $d_\mfU$ distance from $1_{\Ga_\infty}$. In particular, $d_\mfU$ induces the usual topology on $\gainf$.

\paragraph{Identifying $\conga$ with $\Ga_\infty$.} When $\Gamma$ is a torsion free, finitely generated nilpotent group, we shall see that $\conga$ has a natural group law, and that it is isomorphic, as a metric group, to $\gainf$ equipped with some Carnot-Carath\'{e}odory metric (usually different from the metric $d_\mfU$ described above). This is not a new result, but we could not find an explicit reference, so we explain below how to pull it from the literature.

As above, equip $\gahat$ with the word metric associated to some compact generating set, so that inclusion of $\Ga$ into $\gahat$ is a quasi-isometry. Let $i:\conga\To\congahat$ be the induced map on asymptotic cones.

Choose a linear map $L:\Fr{g}\To\Fr{g}_\infty$ such that $L(\Fr{g}_i)=(\Fr{g}_\infty)_i$, and let $\Psi_n:\widehat{\Gamma}\To\Gamma_\infty$ be given by the following compositions:
$$\Psi_n:\widehat{\Gamma}
{\buildrel \text{log}\over\longrightarrow} \mathfrak{g}
{\buildrel \text{L}\over\longrightarrow} \mathfrak{g}_\infty
{\buildrel \delta_{1/n}\over\longrightarrow} \mathfrak{g}_\infty
{\buildrel \text{exp}\over\longrightarrow} \Gamma_\infty.$$
Let $\Psi:\cone(\widehat{\Gamma})\To\Gamma_\infty$ be given by
$$[g_n]\mapsto\lim_{\mfU}\Psi_n(g_n).$$
We have the following facts.
\begin{itemize}
\item By \cite[Proposition 3.1]{cornulier}, one may define a group structure on $\conga$ or $\congahat$ by taking $[g_n][h_n]:\bdef[g_n \: h_n]$, letting the equivalence class of the constant sequence $[1_\Ga]$ be the identity, and taking $[g_n]^{-1}:\bdef[g_n^{-1}]$. Note that this fact is not trivial, as this multiplication need not be well defined for general finitely generated groups $\Ga$. It is clear that the metrics on $\conga$ and $\congahat$ are left-invariant.
\item $\Psi$ is a group isomorphism \cite[Theorem A.9]{cornulier}. Furthermore, if $\gainf$ is equipped with the metric $d_\mfU$ described above, then $\Psi$ is an isometry \cite[Theorem A.9]{cornulier}.
\item The reader may check that the map $i:\conga\To\congahat$ is bilipschitz, and is a group isomorphism. As Lipschitz maps are continuous, it follows that $\Psi\circ i$ is an isomorphism of topological groups (though not of metric groups, in general).
\item It is probably clear that $\conga$ is a geodesic metric space, but we give a proof anyways. It suffices to exhibit, for any $[g_n]\in\conga$, a geodesic path from the identity $[1_\Ga]$ to $[g_n]$. Write $|g_n|$ for $d(g_n,1_\Ga)$, and let $\gamma_n:\{0,\ldots,|g_n|\}\To\Ga$ be a ``discrete geodesic" connecting $1_\Ga$ to $g_n$, so that $\gamma_n(0)=1_\Ga$, $\gamma_n(|g_n|)=g_n$ and $d(\ga_n(a),\ga_n(b))=|b-a|$ for all $a,b\in\{0,\ldots,|g_n|\}$. To define the desired geodesic in $\conga$, let $r=d([g_n],[1_\Ga])$ and take
$$\ga:[0,r]:\To\conga$$
to be given by
$$\ga(t)=[\ga_n(\lfloor(t/r)|g_n|\rfloor)].$$
Alternatively, one may use \cite{Pansu} to see that $\conga$ is isometric to a geodesic metric space.
\item Following \cite{cornulier}, we see that, because the metric on induced by $\Psi\circ i$ on $\gainf$ is geodesic, it must in fact be a Carnot-Carath\'{e}odory metric \cite[Theorem 2.(i)]{Berestovskii}.
\end{itemize}
Henceforth, we shall simply write $\gainf$ for $\conga$, with the understanding that it is equipped with the Carnot-Carath\'{e}odory metric induced by $\Psi\circ i$.

\paragraph{Balls, bounded sets, and Lipschitz maps in the asymptotic cone.} Observe that if $f:\Ga\To\De$ is Lipschitz, we may, by enlarging the generating set of $\De$, assume that it is $1$-Lipschitz. We often do this to save notation.

\begin{lemma}
Let $f:\Gamma\To\Delta$ be a $1$-Lipschitz map. The map
$$f_\infty: \Gamma_\infty\To\Delta_\infty$$
given by
$$[g_n]\mapsto [f(g_n)]$$
is well defined and $1$-Lipschitz.
\end{lemma}

\begin{proof}
See \cite[Proposition 2.9]{cornulier}.
\end{proof}

\begin{lemma}
Let $f:\Gamma\To\Delta$ be a $1$-Lipschitz map, $r$ a positive real number, $(b_n)$ a sequence of elements of $\Gamma$ with $d(b_n,1_G)=O(n)$, and $(S_n),(T_n)$ sequences of subsets of $\Gamma$ with $\diam(S_n)$ and $\diam(T_n)=O(n)$. We have the following equalities of subsets in $\Gamma_\infty$ and $\Delta_\infty$
\begin{itemize}
\item $[S_n][T_n]=[S_n T_n]$, where $S_n T_n:\bdef\{st:s\in S_n,t\in T_n\}.$
\item $[S_n]\cap[T_n]=[S_n\cap T_n]$.
\item $[f(S_n)]=\finf([S_n])$.
\item $[B(r n,b_n)]=\overline{B(r,[b_n])}.$
\end{itemize} 
\end{lemma}

\begin{proof}
The first three items are left as exercises for the reader.

To see the fourth item, first observe that the closure of the $r$-ball in $\Ga_\infty$ is equal to the closed $r$-ball because $\Ga_\infty$ is a complete manifold. To obtain the inclusion $[B(r n,b_n)]\subseteq\overline{B(r,b_n)}$, we argue as follows. Given a sequence $(g_n)\in B(rn,b_n))_{n\in\NN}$, we have that $\lim_\mfU \frac{d(g_n,b_n)}{n}\leq r$. Subsequently, $d([g_n],[b_n])\leq r$, and thus, $[g_n]\in\overline{B(r,[b_n])}$. We now wish to show inclusion in the other direction. If $[g_n]\in \overline{B(r,[b_n])}$, then $\lim_\mfU \frac{d(g_n,b_n)}{n}\leq r$. Define a sequence $(h_n)_{n\in\NN}$ by taking $h_n\in B(rn,b_n)$ to be a nearest point to $g_n$. We have that
$$\lim_\mfU\frac{d(h_n,g_n)}{n}
=\lim_\mfU \frac{d(g_n,b_n)-rn}{n}=0,$$
 so $[g_n]=[h_n]\in[B(rn,b_n)]$ completing the proof of the fourth item and thus the lemma.
\end{proof}


\section{Lipschitz injections of nilpotent groups.}
We are now going to prove the main theorem of the paper. This theorem will state that if $f:\Ga\To\De$ is an injective Lipschitz map between finitely generated nilpotent groups, and $d(\Ga)=d(\De)$, then $\Ga_\infty\cong\De_\infty$. The proof is organized as follows.
\begin{itemize}
\item First, following Pansu \cite{Pansu2}, we see that the induced map on asymptotic cones $f_\infty:\Ga_\infty\To\De_\infty$ is ``Pansu differentiable" almost everywhere. This derivative, where defined, is a group homomorphism $\Ga_\infty\To\De_\infty$.
\item Second, we show that a homomorphism between Carnot groups of the same growth is an isomorphism if it induces a surjective map on abelianizations.
\item Finally, we show that the Pansu derivative of $f_\infty$ at any point where it is defined must be surjective.
\end{itemize}


\subsection{The Pansu derivative.}
\label{subsection:pansuderivative}

Given a Lipschitz map $F:\RR\To\RR$, Rademacher's theorem asserts that $F$ is differentiable almost everywhere. Pansu proved a generalization of this theorem to Lipschitz maps $F:G\To H$ of Carnot groups. In particular, he showed that, given the following definitions, we must have that $F$ is differentiable almost everywhere \cite[Theorem 2]{Pansu2}.

\begin{definition}
Given $g\in G$, we say that $F$ is differentiable at $g$ if the limit
$$\lim_{s\To 0}\delta_{1/s}\left(F(g)^{-1}F(g\:\delta_s(x))\right),$$
converges uniformly for $x$ in any compact subset of $G$. If $F$ is differentiable at $g$, the function
$$x\mapsto \lim_{s\To 0}\delta_{1/s}\left(F(g)^{-1}F(g\:\delta_s(x))\right)
$$
defines a homomorphism known as the derivative of $F$ at $g$ and is denoted $DF|_g:G\To H$. 
\end{definition}

The reader may compare to the ordinary derivative $\lim_{s\To 0}(F(g+sx)-F(g))/s$ of a real valued function of one real variable.

\subsection{Homomorphisms of Carnot groups.}
We will now establish the following alternative: a homomorphism between Carnot groups of the same growth rate either fails to be surjective on abelianizations or is an isomorphism. We break this up into two parts. We first demonstrate that a surjective homomorphism between Carnot groups of the same growth rate is an isomorphism (Proposition \ref{carnot_no_surjective_maps}). We then show that a homomorphism between Carnot groups is surjective if and only if the induced map on abelianizations is surjective (Proposition \ref{abelian_map_equiv}). We then deduce that if a homomorphism $F:\Ga_\infty\To\De_\infty$ of Carnot groups of the same growth is not an isomorphism, its image is annihilated by some nontrivial homomorphism $\ell:\De_\infty\To\RR$ (Corollary \ref{annihilating_functional}).

\begin{prop}\label{carnot_no_surjective_maps}
	Let $\Ga$ and $\De$ be torsion free, finitely generated nilpotent groups such that $d(\Ga) = d(\De)$. Suppose that $\map{F}{\Ga_{\infty}}{\De_\infty}$ is a surjective homomorphism. Then $\Ga_{\infty} \cong \De_\infty$.
\end{prop}
\begin{proof}
	
	Before starting, we establish some notation for this proposition. We let $\Fr{g}$ and $\Fr{h}$ be the Lie algebras of $\Ga_{\infty}$ and $\De_{\infty}$ respectively. We also let $\Ga_{\infty,i} = \pr{\Ga_{\infty}}_i$ and $\De_{\infty,i} = \pr{\De_{\infty}}_i$. Finally, we let $c_1$ and $c_2$ be the step lengths of $\Ga$ and $\De$ respectively. 

Lemma 1.2.5 of \cite{Dekimpe} implies that $\exp(\Fr{g}_t) = \Ga_{\infty,t}$ and $\exp(\Fr{h}_t) = \De_{\infty,t}$. Since $\text{dim}(\Ga_{\infty,t}) = \text{rank}(\Ga_t)$ and $\text{dim}(\De_{\infty,t}) = \text{rank}(\De_t)$, we may phrase the growth rate of $\Ga$ and $\De$ in terms of the $\R$-ranks of the quotients of the lower central series of $\Fr{g}$ and $\Fr{h}$. We may write
	$$
	d(\Ga) = \sum_{k=1}^{c_1} k \cdot \text{dim}_\R (\Fr{g}_k / \Fr{g}_{k+1}).
	$$
	Similarly, we have
	$$
	d(\De) = \sum_{k=1}^{c_2} k \cdot \text{dim}_\R (\Fr{h}_k / \Fr{h}_{k+1}).
	$$
	
	We claim that the induced map $\map{dF_1|_{\Fr{g}_i}}{\Fr{g}_i}{\Fr{h}_i}$ is surjective. We proceed by induction on the term of the lower central series, and since $\Fr{g}_1 = \Fr{g}$ and $\Fr{h}_1 = \Fr{h}$, we have that the base case is evident. Now let $i > 1$, and consider the map $\map{dF_1|_{\Fr{g}_i}}{\Fr{g}_i}{\Fr{h}_i}$. Observe that $\Fr{g}_i = [\Fr{g}_{i-1},\Fr{g}]$ and $\Fr{h}_i = [\Fr{h}_{i-1},\Fr{h}]$.
	
	For each $Y_1 \in \Fr{h}_{i-1}$ and $Y_2 \in \Fr{h}$, the inductive hypothesis implies that there exist $X_1 \in \Fr{g}_{i-1}$ and $X_2 \in \Fr{g}$ such that $dF_1(X_1) = Y_1$ and $dF_1(X_2) = Y_2$. Since $dF_1$ is a Lie algebra morphism, we have $dF_1([X_1,X_2]) = [dF_1(X_1),dF_1(X_2)] = [Y_1,Y_2]$. Observing that $\Fr{h}_i$ is generated by elements of the above form, we are done.
	
	Write $\map{\pi_{\Fr{h}_{i+1}}}{\Fr{h}_i}{\Fr{h}_i/\Fr{h}_{i+1}}$ for the natural projection. Since $dF_1$ is surjective, the map $\pi_{\Fr{h}_{i+1}} \circ dF_1|_{\Fr{g}_i}$ is surjective. Given that $\Fr{g}_{i+1} \leq \ker (\pi_{\Fr{h}_{i+1}} \circ dF_1|_{\Fr{g}_i})$, we have an induced surjective map
	$$
	\map{\widetilde{dF}_1}{\Fr{g}_i / \Fr{g}_{i+1}}{\Fr{h}_i / \Fr{h}_{i+1}}.
	$$
	Therefore, $\text{dim}_\R(\Fr{g}_i / \Fr{g}_{i+1}) \geq \text{dim}_\R(\Fr{h}_i / \Fr{h}_{i+1})$. We may write
	$$
	0 = d(\Ga) - d(\De) = \sum_k k \cdot \pr{\text{dim}_\R(\Fr{g}_i / \Fr{g}_{i+1}) - \text{dim}_\R(\Fr{h}_i / \Fr{h}_{i+1})} \geq 0.$$
	That implies $\text{dim}_\R(\Fr{g}_i / \Fr{g}_{i+1}) = \text{dim}_\R(\Fr{h}_i / \Fr{h}_{i+1})$ for all $i$. In particular, $\text{dim}(\Fr{g}) = \text{dim}(\Fr{h})$. Since $dF_1$ is a surjective map between Lie algebras of the same dimension, it is an isomorphism as desired.
\end{proof}

\begin{prop}\label{abelian_map_equiv}
	Let $\Ga$ and $\De$ be two torsion free, finitely generated  nilpotent groups, and let $F:\Ga_\infty\To\De_\infty$ be a Lie morphism. Then $F$ is surjective if and only if the induced map of abelianizations $F_\text{ab}:(\Ga_\infty)_\text{Ab}\To(\De_\infty)_{\text{Ab}}$ is surjective.
\end{prop}
\begin{proof}
	Let $\Fr{g}$ and $\Fr{h}$ be the Lie algebras of $\Ga_\infty$ and $\Delta_\infty$, respectively. Additionally, let $c_1$ and $c_2$ be the nilpotent step length of $\Ga_\infty$ and $\Delta_\infty$, respectively. We may write the graded decompositions of $\Fr{g}$ and $\Fr{h}$ as
	$$
	\Fr{g} = \bigoplus_{i=1}^{c_1}\Fr{g}_i \quad \text{ and } \quad \Fr{h} = \bigoplus_{i=1}^{c_2} \Fr{h}_i
	$$
	where $[\Fr{g}_1,\Fr{g}_{i}] = \Fr{g}_{i+1}$ and $[\Fr{h}_{1},\Fr{h}_i] = \Fr{h}_{i+1}$ for all $i$.
	Finally, let $dF_\text{Ab}:\Fr{g}_\text{Ab}\To\Fr{h}_\text{Ab}$ be the induced map of abelianizations.
	
	Since $dF_\text{Ab}$ is surjective when $dF$ is surjective, we consider the case of when $dF_{\text{Ab}}:\Fr{g}_\text{Ab}\To\Fr{h}_{\text{Ab}}$ is surjective. We proceed by induction on step length of $\Fr{h}$, and note that the base case follows from assumption. Thus, we may assume that $c_2 > 1$. Observe that $\Fr{h}_\text{ab} \cong (\Fr{h} / \Fr{h}_{c_2})_{\text{ab}}$. That implies the abelianization of the induced map $\widetilde{dF}:\Fr{g}\To\Fr{h} /  \Fr{h}_{c_2}$ is equivalent to the map $dF_\text{Ab}:\Fr{g}_{\text{ab}}\To\Fr{h}_{\text{Ab}}$. Therefore, the inductive hypothesis implies the map $\widetilde{dF}:\Fr{g}\To\Fr{h} / \Fr{h}_{c_2}$ is surjective. For $Y \in \Fr{h}_{c_2}$, there exists $Y_1 \in \Fr{h}_{1}$ and $Y_2 \in \Fr{h}_{c_2 - 1}$ such that $Y = [Y_1,Y_2]$. There exists $X_1, X_2 \in \Fr{g}$ such that $\widetilde{F}(X_1) = Y_1 \: \text{ mod } \: \Fr{h}_{c_2}$ and $\widetilde{F}(X_2) = Y_2 \: \text{ mod } \: \Fr{h}_{c_2}$. Thus, there exists $Z_1,Z_2 \in \Fr{h}_{c_2}$ such that $F(X_1) = Y_1 + Z_1$ and $F(X_2) = Y_2 + Z_2$. We may write
	$$
	dF([X_1,X_2]) = [dF(X_1),dF(X_2)] = [Y_1 + Z_1, Y_2 + Z_2] = [Y_1,Y_2] + [Y_1,Z_2] + [Z_1,Y_2] + [Z_1,Z_2].
	$$
	Since $\Fr{h}$ is Carnot, $\Fr{h}_{c_2} = Z(\Fr{h})$. Thus, $dF([X_1,X_2]) = [Y_1,Y_2] = Y$. In particular, $dF$ is surjective. \end{proof}

Combining the previous propositions, we have the following alternative.

\begin{corollary}\label{annihilating_functional}
Suppose $d(\Gamma)=d(\Delta)$, and $F:\Gamma_\infty\To\Delta_\infty$ is a homomorphism of Carnot groups. Either $F$ is an isomorphism, or there exists a homomorphism $\ell:\Delta_\infty\To\RR$ such that $\ell\circ F=0$ and $\sup\{\ell(h):h\in B(1,1_{\Delta_\infty})\}=1$.
\end{corollary}

\begin{proof}
Suppose $F$ is not an isomorphism. Proposition \ref{carnot_no_surjective_maps} implies that $F$ is not surjective, and Proposition \ref{abelian_map_equiv} implies that $F_{\text{Ab}}$ is not surjective. Thus, $\text{Im}(F_{\text{Ab}})$ is a proper subspace of $(\De_{\infty})_{\text{Ab}}$ which is a finite dimensional $\R$-vector space. Hence, we may choose a norm on $(\Delta_\infty)_\text{Ab}$ and a unit-norm vector $v$ orthogonal to $\text{Im}(F_{\text{Ab}})$.  Let $\map{\ell_v}{(\De_{\infty})_{\text{Ab}}}{\R}$ be given by taking dot product with $v$ so that $\ell_v(v) = 1$ and $\ell_v(\text{Im}(F_{\text{ab}})) = 0$. Let $\pi_{\text{Ab}}:\De_{\infty}\To(\De_{\infty})_{\text{Ab}}$ be the natural projection, and let $M = \text{sup}\set{\ell_v(\pi_{\text{Ab}}(h)) \: | \: h \in B(1,1_{\De_{\infty}})}$ which is non-zero. Thus, the map $\map{\ell}{\De_{\infty}}{\R}$ given by $\ell(h) = M^{-1}\ell_v(\pi_{\text{Ab}}(h))$ satisfies the conditions of the proposition.
\end{proof}

\subsection{Applying Pansu's theorem.}
\label{subsection:maintheorem}
We now prove Theorem \ref{introtheorem}. 

\begin{theorem}
\label{maintheorem}
If $f:\Gamma\To\Delta$ is a Lipschitz injection of torsion free, finitely generated nilpotent groups, and $d(\Gamma)=d(\Delta)$, then $\Gamma_\infty\cong\Delta_\infty$.
\end{theorem}

\begin{proof}
Let $f_\infty:\Gamma_\infty\To\Delta_\infty$ be the induced map on asymptotic cones. Towards a contradiction, assume that $\Gamma_\infty$ and $\Delta_\infty$ are not isomorphic. We further assume, without loss of generality, that $f$ is $1$-Lipschitz.

The idea of the proof is as follows.
\begin{itemize}
\item First, we use Pansu's theorem (\S\ref{subsection:pansuderivative}) to show that there exists $x\in\Gamma_\infty$ such that $f_\infty$ is differentiable at $x$. Thus, we may choose some nonzero homomorphism $\ell:\Delta_\infty\To\RR$ which annihilates this derivative.
\item Second, we show that for every $N$ there exists $\epsilon>0$ such that $\overline{B(2\epsilon,f_\infty(x))}$ contains at least $N$ disjoint translates of $f_\infty(\overline{B(\epsilon,x)})$. In particular, we find these translates by moving in a direction ``perpendicular" to the kernel of $\ell$, so that they have disjoint $\ell$-images.
\item Finally, we take a sequence $b_n$ in $\Gamma$ such that $[b_n]=x$ and conclude that for every $N$ there is some $\epsilon>0$ such that along $\mfU$ the set $B(2\epsilon n,f(b_n))$ contains at least $N$ disjoint translates of $f(B(\epsilon n,b_n))$. Since $f$ is injective, this contradicts our assumption that $\Gamma$ and $\Delta$ have the same growth.
\end{itemize}

\paragraph{Constructing $x$ and $\ell$.} By Pansu's theorem (\S\ref{subsection:pansuderivative}), the $1$-Lipschitz function $f_\infty:\Gamma_\infty\To\Delta_\infty$ is differentiable almost everywhere, so there exists some $x\in\Gamma_\infty$ where it is differentiable. Write $Df_\infty$ for $Df_\infty|_{x}$. By Corollary \ref{annihilating_functional}, we may fix a homomorphism $\ell:\Delta_\infty\To\RR$ such that $\ell\circ Df_\infty=0$ and $\sup\{\ell(h):h\in \overline{B(1,1_{\Delta_\infty})}\}=1$. In what follows, the reader may wish to make the simplifying assumption that $x=1_{\Gamma_\infty}$ and $f_\infty(x)=1_{\Delta_\infty}$. (It is an exercise that this loses no generality).

\paragraph{Finding disjoint translates of $f_\infty(\overline{B(\epsilon,x)})$.} For $\epsilon>0$, let
$$\mu(\epsilon)=3\sup\{|\ell(\finf(g))-\ell(\finf(x))|:g\in \overline{B(\epsilon,x)}\}.$$
Because the defining limit for $Df_\infty$ converges uniformly on compact subsets and $\ell$ is Lipschitz, we have $\meps/\epsilon\To 0$ as $\epsilon\To 0$. Choose $h\in \overline{B(1,1_{\Delta_\infty})}$ such that $\ell(h)=1$, and let $\heps=\finf(x)\:\delta_{\meps}(h)\:\finf(x)^{-1}$, so that $\ell(\heps)=\meps$.

Consider the set $\Aeps$ of translates of $\finf(\overline{B(\epsilon,x)})$ given by
$$\Aeps=\left\{
(\heps)^j \finf(\overline{B(\epsilon,x)}):j\in\ZZ;|j|<\frac{\epsilon}{\meps}
\right\}.$$
We will simply refer to elements of $\Aeps$ as $\epsilon$-\textit{good translates} and shall establish the following.

\begin{itemize}
\item $\epsilon$-good translates are subsets of $\overline{B(2\epsilon,\finf(x))}$.
\item $\epsilon$-good translates are disjoint, i.e., if $B_1,B_2\in \Aeps$ with $B_1\neq B_2$, then $B_1\cap B_2=\emptyset$.
\item The number of $\epsilon$-good translates is $1+2\lfloor\frac{\epsilon}{\meps}\rfloor$ (i.e., $\#\Aeps=\#\{j\in \ZZ:|j|<\epsilon/\meps\}$).
\end{itemize}

\paragraph{$\epsilon$-good translates are subsets of the $2\epsilon$ ball.} Observe that for any integer $j$,
$$(\heps)^j\:\finf(\overline{B(\epsilon,x)})\subseteq (\heps)^j\: B(\epsilon,\finf(x))
=\left(\finf(x)\:\delta_\meps(h)^j\:\finf(x)^{-1}\right)\:
\finf(x)\:B(\epsilon,1_{\Delta_\infty})$$
$$=\finf(x)\:\delta_\meps(h)^j\:B(\epsilon,1_{\Delta_\infty}) \subseteq \finf(x)\:B(|j|\mu(\epsilon),1_{\Delta_\infty})\:
B(\epsilon,1_{\Delta_\infty})=B(\epsilon+|j|\meps,\finf(x)).$$
It follows that $\epsilon$-good translates are subsets of $\overline{B(2\epsilon,\finf(x))}$.

\paragraph{$\epsilon$-good translates are disjoint.} Since $\mu(\epsilon)=3\sup\{|\ell(\finf(g))-\ell(\finf(x))|:g\in \overline{B(\epsilon,x)}\},$ we have for $j \in \ZZ$ that
$$\ell((\heps)^j \finf(\overline{B(\epsilon,x)}))  
=j\meps + \ell(\finf(\overline{B(\epsilon,x)}))$$
$$\subseteq
\left(\left(j-\frac{1}{3}\right)\meps+\ell(\finf(x)), \left(j+\frac{1}{3}\right)\meps+\ell(\finf(x))\right).$$
Hence, for $j,k\in\ZZ$ with $j\neq k$ we have that
$$(\heps)^j \finf(\overline{B(\epsilon,x)})
\cap (\heps)^k \finf(\overline{B(\epsilon,x)})
=\emptyset$$
because the images under $\ell$ of $(\heps)^j \finf(\overline{B(\epsilon,x)})$ and $(\heps)^k \finf(\overline{B(\epsilon,x)})$ are contained in disjoint intervals.

\paragraph{Counting $\epsilon$-good translates.} Observe that there are exactly $1+2\lfloor\frac{\epsilon}{\meps}\rfloor$ integers $j$ such that $j\meps\leq 2\eps$. Hence, $\#\Aeps=1+2\lfloor\frac{\epsilon}{\meps}\rfloor$, and thus for any $N>0$ we have that $\# A_\epsilon>N$ for sufficiently small $\epsilon>0$, since $\epsilon/\meps\To\infty$ as $\epsilon\To 0$.

\paragraph{Finding disjoint translates of $f$-images of balls in $\Gamma$.} Take a sequence $(b_n)_{n\in\NN}$ of elements of $\Gamma$ with $[b_n]=x$, and, for each $\epsilon>0$, a sequence $(\hneps)_{n\in\NN}$ in $\Delta$ such that $[\hneps]=\heps$. For an integer $j$, we have that $[(\hneps)^j f(B(\epsilon n,b_n))]=(\heps)^j \finf(\overline{B(\epsilon,x)})$ and
$[B(2\epsilon n,f(b_n))]=\overline{B(2\epsilon,\finf(x))}$.

For any $\epsilon>0$ and $n\in\NN$, consider the set of translates of $f(B(\epsilon n,b_n))$ given by
$$\Aneps=
\left\{
(\hneps)^j f(B(\epsilon n,b_n)) : j\in\ZZ;|j|<\frac{\epsilon}{\meps}
\right\}.$$ 
We call these $(\epsilon,n)$-\textit{good translates}. We see for any $\epsilon$ that the following properties of $\Aeps$ pass to $\Aneps$ along $\mfU$, i.e., there exists $U\in\mfU$ such that the following statements all hold for $n\in U$.
\begin{itemize}
\item Each $(\epsilon,n)$-good translate is a subset of $B(2\epsilon n,f(b_n))$. 
\item $(\epsilon,n)$-good translates are disjoint, i.e., if $B_1,B_2\in \Aneps$ with $B_1\neq B_2$, then $B_1\cap B_2=\emptyset$.
\item The number of $(\epsilon,n)$-good translates is $1+2\lfloor\frac{\epsilon}{\meps}\rfloor$.
\end{itemize}

\paragraph{Conclusion.} Observe that, since $\Gamma$ and $\Delta$ both have growth of order $n^{d(\Gamma)}$, there exists $N>0$ such that
$$N\:\#B(n,1_\Gamma)>\# B(2n,1_\Delta)$$
for all $n>N$. Take $\epsilon > 0$ such that $\#\Aneps>N$. Then along $\mfU$,
$$\bigsqcup_{B\in\Aneps}B\subseteq B(2\epsilon n, f(b_n)),$$
and, thus
$$\#\bigsqcup_{B\in\Aneps}B\leq\# B(2\epsilon n, f(b_n)).$$
On the other hand, by injectivity of $f$ we see that 
$$\# \bigsqcup_{B\in\Aneps}B = N(\# B(\epsilon n,1_\Gamma))>B(2\epsilon n,f(b_n)).$$
We have thus obtained a contradiction, as desired.
\end{proof}

\bibliographystyle{plain}
\bibliography{bibliography}
\end{document}